\documentclass[final]{siamltex}
\usepackage{algorithm,algpseudocode}
\usepackage{latexsym, graphicx, epsfig, amsmath, amsfonts,amssymb,bm}
\usepackage{epstopdf}
\usepackage{booktabs}
\usepackage{array}
\usepackage{color}

\def\be{\begin{equation}}
\def\ee{\end{equation}}

\def\w{\omega}

\def\x{\mathbf{x}}
\def\w{\mathbf{w}}

\def\z{\mathbf{z}}

\def\R{{\mathbb R}}

\title{An Explicit Neural Network Construction for Piecewise Constant Function Approximation}

\author{Kailiang Wu \and
 Dongbin
       Xiu\thanks{Department of Mathematics,
		The Ohio State University, Columbus, OH 43210, USA.
		{\tt wu.3423@osu.edu, xiu.16@osu.edu.}
		}
}

\begin{document}
\maketitle

\begin{abstract}
We present an explicit construction for feedforward neural network (FNN), which provides a piecewise constant
approximation for multivariate functions. The proposed FNN has two hidden layers, where the weights and thresholds are explicitly defined 
and do not require numerical optimization for training. Unlike most of the existing work on explicit FNN construction, the proposed FNN
does not rely on tensor structure in multiple dimensions. Instead, it automatically creates Voronoi tessellation of the domain, based on the
given data of the target function, and piecewise constant approximation of the function. This makes the construction more practical for applications. 
We present both theoretical analysis and numerical examples to demonstrate its properties.
\end{abstract}


% main text
%\input Introduction

\section{Introduction} \label{sec:intro}

Feedforward neural network (FNN) has attracted wide attention in
recent years, largely due to the many successes it brings to machine
learning and artificial intelligence. There are an exceedingly large
number of literature devoted to various aspects of FNN, particularly
on their performance and algorithm design.

Most of the existing mathematical studies on FNN focus on single-hidden-layer
FNN and its ability to approximate unknown target functions. The
earlier theoretical results established that single-hidden-layer FNN
is capable for approximating functions with arbitrary accuracy, see, for example,
\cite{Cybenko89, Pinkus99, Barron93}.
Efforts have then been made to explicitly constructive single-hidden-layer
FNNs, where the weights and
thresholds are explicitly specified and not solved numerically by a
certain optimization procedure.
These explicitly defined single-hidden-layer FNNs provide very useful,
from the mathematical perspective, constructive
proofs for the existence of FNNs for function approximation.
One of the earliest constructions is the work in
\cite{CardaliaguetEuvrard92}. Since then, several other constructions
have been presented, many of which are based on variations of the
Cardaliaguet-Euvrard operator from \cite{CardaliaguetEuvrard92}. 
These work include \cite{Anastassiou1997, CaoZhangHe09, Costarelli2015,
	CosterelliSpigler2013, LlanasSainz06}, to name a few. A common
feature of these work is that the construction is typically based on a
univariate formulation, for example, the Cardaliaguet-Euvrard
operator \cite{CardaliaguetEuvrard92} in 1D. To generalize to multivariate functions, tensor product 
is employed. The complexity of the constructions thus grows
exponentially in high dimensions. Consequently, although these results are
useful from the mathematical view point, they do not provide practical tools for applications.

In this paper, we present a new explicit construction of FNN
for multivariate function approximation. A distinct feature of the
proposed construction is that it does not utilize tensor structure in
multiple dimensions. Instead, the network creates a piecewise constant
approximation for any given function based the available data. The
construction uses exclusively the threshold function, also known as hard
limiter or binary function, as the activation function.
The
weights and thresholds in the network are explicitly defined, and
we prove that the ``pieces'' in the piecewise constant approximation
form a Voronoi diagram (cf. \cite{Okabe1992}) of the
domain. Suppose one is given $n$ sample data of the unknown target
function, the proposed FNN then provides a piecewise constant
approximation based $n$ Voronoi cells of the underlying domain. We
also provides an error estimate of the approximation.
Due to the explicit construction, the weights and thresholds of the FNN can
be easily evaluated, thus avoiding a potentially expensive numerical
optimization procedure for their training. This, along with the
non-tensor structure of the construction, makes the proposed FNN a
practical tool for
applications, as it works with arbitrarily given data in arbitrary
dimensions.
We demonstrate the effectiveness of the FNN using
several numerical examples, which also verify the error convergence
from the theoretical estimate.
In the current construction, which is perhaps one of the most
intuitive ones, the network uses two hidden layers with (at most) $n^2$ neurons.
Similar FNNs
with less number of neurons are possible, by using more involved
network structures. This shall be investigated in future studies.
%Unlike most of the existing
%FNN constructions with single hidden layer, our construction includes two
%hidden layers and does not use tensor structure for multivariate
%functions. 
The proposed FNN is not only another (and new) constructive proof for
the universal approximation property of FNNs, but also a practical
tool for real data.

This paper is organized as follows. Upon a brief setup of the problem
in Section \ref{sec:setup}, we present the explicit FNN construction
in Section \ref{sec:method}, which also includes its theoretical
analysis. Numerical examples are then presented in Section
\ref{sec:examples} to demonstrate the properties of the FNN.

\section{Setup} \label{sec:setup}

Consider the problem of approximating an unknown function
$f:D\to\R$ using its samples, where $D\subseteq \R^d$, $d\geq 1$.
%and is
%equipped with a measure $\w$. 
Let ${\x}=(x_1,\dots, x_d)$ be the
coordinate and
\begin{equation} \label{samples}
\left(\x^{(1)}, f^{(1)}\right), \dots, \left(\x^{(n)}, f^{(n)}\right),
\qquad n>1,
\end{equation}
be a set of training data, where
$\x^{(k)}\in D$ are the location of the data samples and 
\begin{equation*} 
f^{(k)} = f( {\x}^{(k)})+\epsilon_k,  \qquad k=1,\dots, n,
\end{equation*} 
are the sample data, with $\epsilon_k\geq 0$ being (possible) observation error. We
consider only the nontrivial case of $n>1$. 
%Throughout this paper, we
%use $\|\cdot\|$ to denote vector 2-norm, unless otherwise noted.

%\input Method

\section{Construction of the FNN} \label{sec:method}

In this section, we present our construction of the feedforward neural
network (FNN) with two hidden layers. We first present the detail of
its structure and then prove that it provides a piecewise constant
approximation for any target function $f:\R^d\to\R$.

\subsection{The construction}

The structure of the FNN is illustrated in
Figure \ref{fig:NN4}. It consists of an input layer, which has
$d\geq 1$ neurons corresponding to the
$d$-variate input signal, and an output layer with one neuron, as the
function under consideration is
$f:\R^d\to\R$. This is the standard setup for most of the FNN.

We will use exclusively the threshold function, also known as the step
function, hard limiter function, binary function, etc, as the activation function
\begin{equation}  \label{step}
s(x) = \left\{
\begin{array}{ll}
1, & x\geq 0,\\
0, & x<0.
\end{array}
\right.
\end{equation}

\begin{figure} 
	\centering
	\includegraphics[width=0.99\textwidth]{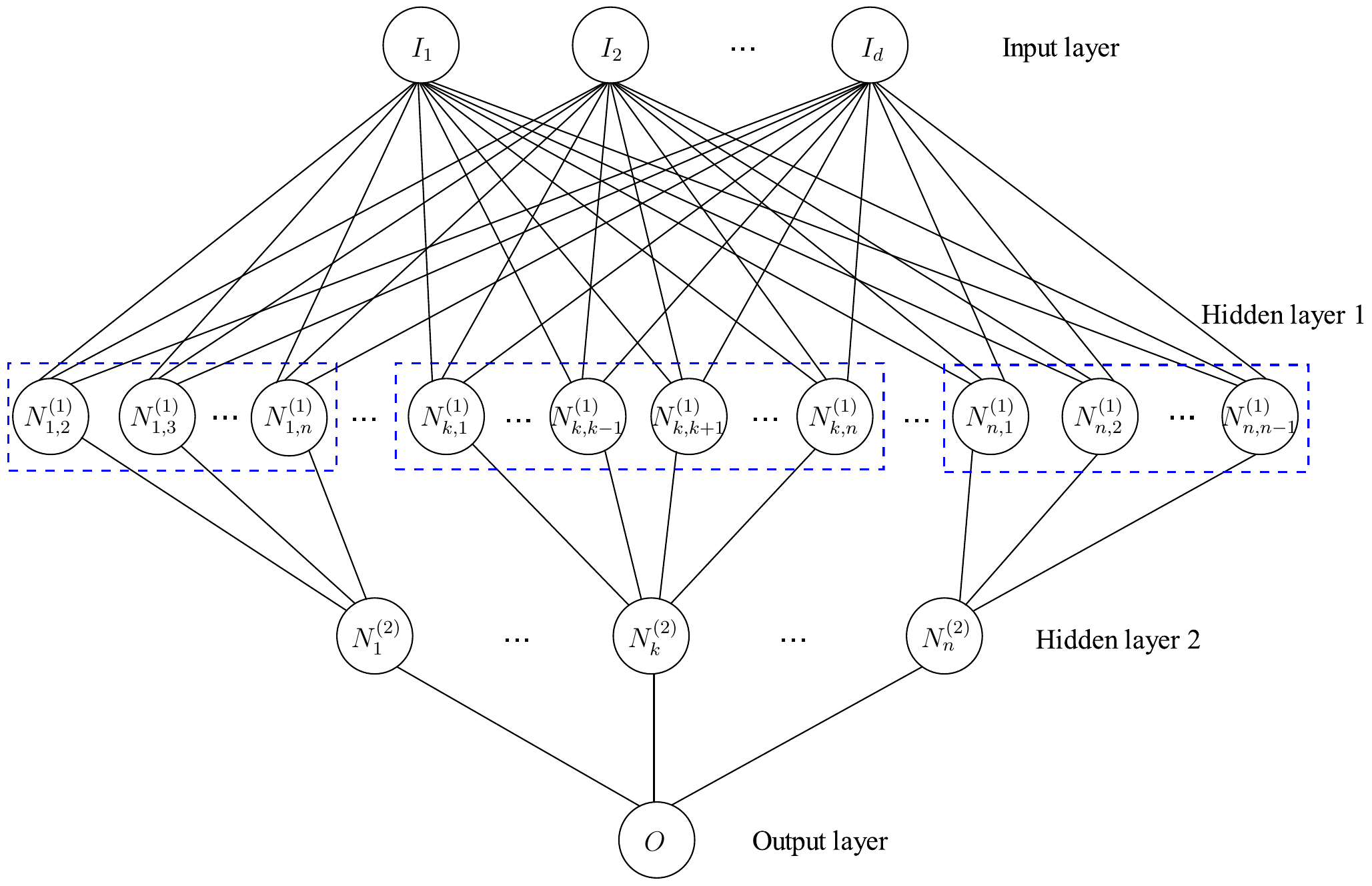}
	\caption{The schematic diagram of the FNN with two hidden
		layers.}
	%with three input neurons, two hidden layers consisting of $n(n-1)$ and $n$ neurons respectively,
	%		and one output neuron. 
	%		The biases in input layer and the first hidden layer are omitted in this figure.  
	%		Within each connection, information flows from up to bottom.}
	\label{fig:NN4}
\end{figure}

\subsubsection{First hidden layer}

The first hidden layer consists of $n(n-1)$ neurons, where $n>1$ is
the number of training samples. We shall divide the neurons into $n$
groups, with each group corresponding to a training sample. Each group
then includes $(n-1)$ neurons, thus making the total number of neurons
in the first hidden layer
$n(n-1)$.
We shall use $N^{(1)}$ to denote the neurons in the first hidden layer
and label them with two indices in the following way,
$$
N^{(1)}_{k,j}, \qquad 1\leq k \leq n, \quad 1\leq j\leq n, 
\quad j\neq k,
$$
where the first index $k=1,\dots, n$, denotes the $k$-th group, and
the second index $j\neq k$ denotes its location within the $k$-th
group. 
Note that this is merely an indexing scheme to distinctly identify the
neurons. The neurons have no lateral connections and each receive the
same signals from the input layer, as illustrated in Fig. \ref{fig:NN4}.

The output of the neurons in the first hidden layer is explicitly defined as follows,
\begin{equation} \label{z1}
z^{(1)}_{k,j} (\x) = s(\w_{k,j}\cdot \x - b_{k,j}), \qquad 1\leq k
\leq n, \quad 1\leq j\leq n, \quad j\neq k,
\end{equation}
where
\begin{equation} \label{w1}
\begin{split}
\w_{k,j}  & = \x^{(k)} - \x^{(j)}, \\
b_{k,j} &= \frac{1}{2} (\x^{(k)} - \x^{(j)})\cdot (\x^{(k)} + \x^{(j)}).
\end{split}
\end{equation}
Here $\x^{(k)}$ are the coordinates of the $k$-th training sample
\eqref{samples}. Obviously, $z_{k,j}^{(1)}\in \{0,1\}$ because of the
use of the binary activation function. We write
\begin{equation} \label{z1vector}
%\z_k^{(1)}(\x) = (z_{k,1}^{(1)}, \dots, z_{k,n}^{(1)})^\top 
{\z_k^{(1)}(\x) = \left(\dots, z_{k,k-1}^{(1)}, z_{k,k+1}^{(1)},\dots \right)^\top} 
\in \ \{0,1\}^{n-1}, \qquad 1\leq k\leq n,
\end{equation}
as the output vector of the $k$-th group, $1\leq k\leq n$.

\subsubsection{Second hidden layer}

The second hidden layer consists of $n$ neurons, each of which only receives 
the output signals from a unique group in the first hidden layer, as
illustrated in Figure \ref{fig:NN4}. 
The output of each neuron in the second hidden layer is defined as,
\begin{equation} \label{z2}
z_k^{(2)}(\x) = s(\mathbf{1}\cdot \z_k^{(1)}(\x) - b^{(2)}), \qquad
k=1,\dots, n,
\end{equation}
where
\begin{equation} \label{w2}
\mathbf{1} = (1,\dots,1)^\top \in\R^{n-1}, \qquad
b^{(2)} = n-1,
\end{equation}
and $\z_k^{(1)}(\x)$ is the output vector from the $k$-th group in the
first layer, as defined in \eqref{z1vector}.
To avoid the unexpected effect of computer round-off error, during
implementation one may
define $b^{(2)} = n-1-\epsilon$, where $0<\epsilon<1$ is an arbitrary
constant. It is obvious $z_k^{(2)}(\x)\in\{0,1\}$, $k=1,\dots,n$.

\subsubsection{Output layer}

The output layer receives signals from all the neurons in the second  hidden 
layer, as shown in Figure \ref{fig:NN4}. It produces the final output
as follow,
\begin{equation} \label{output}
y(\x) 
%= \z^{(2)}\cdot \f 
= \sum_{k=1}^n f^{(k)} \cdot z_k^{(2)}(\x),
\end{equation} 
where $f^{(k)}$, $k=1,\dots, n$, are the sample data.

\subsection{Approximation property}

It is straightforward to show that the outputs of the two hidden
layers \eqref{z2} effectively
produce a Voronoi diagram for the underlying domain $D$, where the
function $f$ is defined, and the final output \eqref{output} thus
becomes a piecewise constant approximation of $f$ based on the Voronoi diagram.
To proceed, we first invoke the concept of (ordinary) Voronoi
diagram (cf. \cite{Okabe1992}).
\begin{definition}[Voronoi diagram]
	For points $X=\{\x^{(1)},\dots, \x^{(n)}\}\subset D \subseteq \R^d$,
	where $2\leq n < \infty$ and $\x^{(i)} \neq \x^{(j)}$ for $i\neq
	j$. We call the region
	\begin{equation}\label{Voronoi}
	V^{(i)} = \left\{\x \mid \|\x-\x^{(i)}\|_2 \leq \|\x-\x^{(j)}\|_2,
	\quad \forall j\neq i\right\}
	\end{equation}
	the ordinary Voronoi cell associated with $\x^{(i)}$, and the set
	$\mathcal{V}(X)= \{V^{(1)},\dots, V^{(n)}\}$ the ordinary
	Voronoi diagram generated by $X$.
\end{definition}

\begin{theorem}
	The output of the FNN \eqref{output} is a piecewise constant
	approximation to the function $f(x)$ using the training data set
	$\left\{\x^{(k)}, f(\x^{(k)})\right\}_{k=1}^n$. More precisely, for $n\geq 2$,
	\begin{equation} \label{f_pw}
	y(\x) = \sum_{k=1}^n f(\x^{(k)}) \mathbb{I}_{V^{(k)}}(\x),
	\end{equation}
	where 
	$V^{(k)}$ is the Voronoi cell associated by the point $\x^{(k)}$, and
	$\mathbb{I}_A(\x)$ is the indicator function for a set $A$ satisfying
	$$
	\mathbb{I}_A(\x) = \left\{
	\begin{array} {ll}
	1, & \x\in A,\\
	0, & \x\notin A.
	\end{array}
	\right.
	$$
	%For the trivial case $n=1$, $y(\x) = f(\x^{(1)})=const$.
\end{theorem}
\begin{proof}
	For earch $\x^{(k)}$, $k=1,\dots, n$, $n\geq 2$, consider $\x^{(j)}$,
	$j\neq k$.
	Let
	$$
	\x^{mid}_{k.j} = \frac{1}{2}(\x^{(k)} + \x^{(j)})
	$$
	be the mid-point between the two points $\x^{(k)}$ and
	$\x^{(j)}$. Then, the output of the first hidden layer \eqref{z1}
	$z_{k,j}^{(1)} = 1$, if and only if,
	$$
	(\x^{(k)} - \x^{(j)})\cdot \x \geq  (\x^{(k)} - \x^{(j)})\cdot
	\x^{mid}_{j,k},
	$$
	which is equivalent to
	\begin{equation} \label{closer} 
	(\x-\x^{mid}_{j,k} )\cdot (\x^{(k)} - \x^{(j)})\geq 0.
	\end{equation}
	Note that
	$$
	\{ \x \mid (\x-\x^{mid}_{j,k} )\cdot (\x^{(k)} - \x^{(j)})= 0\}
	$$
	is the center hyperplane separating the two points $\x^{(k)}$ and
	$\x^{(j)}$. Therefore, \eqref{closer} contains all the points that are
	closer to $\x^{(k)}$ than to $\x^{(j)}$. We then have, for each $1\leq
	k\leq n$ and $j\neq k$, 
	$$
	\left\{\x \mid z_{k,j}^{(1)}(\x) = 1\right\} =
	\left\{\x \mid \|\x-\x^{(k)}\|_2\leq \|\x-\x^{(j)}\|_2\right\}.
	$$
	
	The output of the $k$th neuron in the second hidden layer \eqref{z2}
	satisfies
	$$
	z_k^{(2)} = 1, \quad\textrm{if and only if,}\quad
	\z_k^{(1)} = \mathbf{1},
	$$
	where again $\mathbf{1} = (1,\dots,1)^T \in\R^{n-1}$ is a vector of
	length $(n-1)$. This is equivalent to $z_{k,j}^{(1)}=1$, $\forall
	j\neq k$. Therefore,
	$$
	\left\{\x \mid z_{k,j}^{(2)}(\x) = 1\right\} =
	\left\{\x \mid \|\x-\x^{(k)}\|_2 \leq \|\x-\x^{(j)}\|_2, \forall
	j\neq k\right\} = V^{(k)},
	$$
	which is, by definition, the Voronoi cell associated with the point
	$\x^{(k)}$.
	The output of the network \eqref{output} is obviously a piecewise
	constant function in the form of \eqref{f_pw}.
\end{proof}

If one assumes certain differentiability condition on the target function $f$,
then  we have the following result on the error estimate in $L^2$
norm. More specifically, we define the $L^2$ norm with respect to a
measure $\mu(\x)$,
\begin{equation}
\|f\|_{L^2_{d\mu}(D)} := \left(\int_D f^2(\x) d\mu(\x)\right)^{1/2}.
\end{equation}
We assume that the volume of $D$ with respect to $\mu$ is a finite
constant, and without loss of generality, we set this constant to be
1. That is,
$$
\int_D d\mu(\x) =1.
$$
\begin{theorem} \label{thm:error}
	Assume $f({\bf x})$ is
	differentiable and with bounded first-order derivatives, and
	the measure $\mu$ is such that
	$$
	\int_D \|\x\|_2^2 d\mu(\x) < \infty.
	$$
	Then, the approximation error of the FNN \eqref{output} satisfies
	\begin{align}\label{eq:L2errnew}
	\left\| y-f \right\|_{L_{d\mu}^2(D)} \le 
	\left( \mathop {\sup }\limits_{ {\bf x} \in D } \big\| \nabla f ({\bf x})\big\|_2 \right) \left( \sum_{i=1}^n \int_{V^{(i)}} \big\| {\bf x}^{(i)} - {\bf x} \big\|_2^2  d \mu({\bf x}) \right)^{\frac12}.
	%\delta \left( \sup_{{\bf x}\in D} \| \nabla f ({\bf x}) \|_2 \sqrt{ \mu (D) } \right),
	\end{align}
\end{theorem}

\begin{proof}
	From \eqref{f_pw}, we have
	\begin{equation*}
	\begin{split}
	\left\| y-f \right\|_{L_{d\mu}^2(D)}^2 
	& =   \sum_{i=1}^n  \int_{V^{(i)}} \big| f({\bf x}^{(i)}) - f({\bf x}) \big|^2
	d \mu({\bf x}) \\
	&= \sum_{i=1}^n  \int_{V^{(i)}} \bigg | \int_0^1  ( {\bf x}^{(i)} - {\bf x} ) \cdot \nabla f( t {\bf x}^{(i)} + ( 1- t ) {\bf x}) dt \bigg|^2 d \mu({\bf x})
	\\
	& \le \sum_{i=1}^n  \int_{V^{(i)}}  \bigg( \big\| {\bf x}^{(i)} - {\bf x} \big\|_2  \int_0^1  \big\| \nabla f( t {\bf x}^{(i)} + ( 1- t ) {\bf x}) \big\|_2   dt \bigg)^2 d \mu({\bf x})
	\\
	& \le \sum_{i=1}^n \left( \mathop {\sup }\limits_{ {\bf x} \in V^{(i)} }  \big\| \nabla f ({\bf x})\big\|_2 \right)^2 \int_{V^{(i)}} \big\| {\bf x}^{(i)} - {\bf x} \big\|_2^2  d \mu({\bf x})
	\\
	&= \left( \mathop {\sup }\limits_{ {\bf x} \in D } \big\| \nabla f ({\bf x})\big\|_2 \right)^2 \sum_{i=1}^n \int_{V^{(i)}} \big\| {\bf x}^{(i)} - {\bf x} \big\|_2^2  d \mu({\bf x}).
	\end{split}
	\end{equation*}
	The proof is then completed.
	%which implies \eqref{eq:L2err}. %Similar arguments yield \eqref{eq:Linferr}.  
\end{proof}

\begin{corollary}
	Under the hypotheses of Theorem \ref{thm:error}, if $D$ is bounded and 
	$\mu$ is uniform measure, then 
	\begin{align}\label{eq:L2err}
	\left\| y-f \right\|_{L_{d\mu}^2(D)} \le 
	\left( \sup_{{\bf x}\in D} \| \nabla f ({\bf x}) \|_2 \right) \delta,
	\end{align}	
	where 
	$$
	\delta := \max_{1\le i \le n} {\rm diam} (V^{(i)}),
	$$
	with
	${\rm diam}(A):=\sup_{{\bf x},{\bf
			x}' \in A} \|{\bf x}-{\bf x}'\|_2$ denoting the diameter of a bounded set $A$.	
\end{corollary}

\begin{proof}
	The conclusion directly follows from Theorem \ref{thm:error}, because 
	\begin{equation*}
	\sum_{i=1}^n \int_{V^{(i)}} \big\| {\bf x}^{(i)} - {\bf x} \big\|_2^2  d \mu({\bf x}) \le \delta^2  \sum_{i=1}^n \int_{ V^{(i)} }  d \mu({\bf x}) =  \delta^2.
	\end{equation*}
\end{proof}

Note that, if the training data points $\{ {\bf x}^{(j)} \}_{j=1}^n$
are (almost) uniformly distributed in the bounded domain $D$, then 
$
\delta \sim {n^{-\frac1d}},
%=\big( n^2 \big)^{-\frac1{2d}}
$
c.f., \cite{devroye2017measure}. 
This implies that the error of our FNN scales as 
\begin{equation} \label{rate}
\left\|
y-f \right\|_{L_{d\mu}^2(D)} \sim {\mathcal O}( {n^{-\frac1d}}).
\end{equation}

\section{Numerical Examples} \label{sec:examples}

In this section we present numerical examples to demonstrate the
properties of our FNN.
Upon explicitly constructing the FNN using training data, all numerical errors
are computed using another set of samples --- a validation sample set. In all our examples, the
validation set consists of $M$ randomly generated points that are
independent of the training set. 
The size $M$ is taken as $200$ 
and $10,000$, for univariate
and multivariate tests, respectively.
We compute the differences between the
FNN approximations and the true functions on the validation sets and
report both their $\ell_\infty$ vector norm and $\ell_2$ vector norm,
denoted as $\epsilon_\infty$ error and $\epsilon_2$ error, respectively.

\subsection{Univariate functions}

We first consider a simple smooth function $f(x) = \sin(4\pi x)$,
$x\in [0,1]$. Fig. \ref{fig:Cosine_1d} shows the FNN approximation of
this function with $n=32$ randomly generated training samples. The
numerical approximation by the FNN is clearly a piecewise constant
approximation of the exact function.
\begin{figure}[htbp]
	\centering
	{\includegraphics[width=10cm]{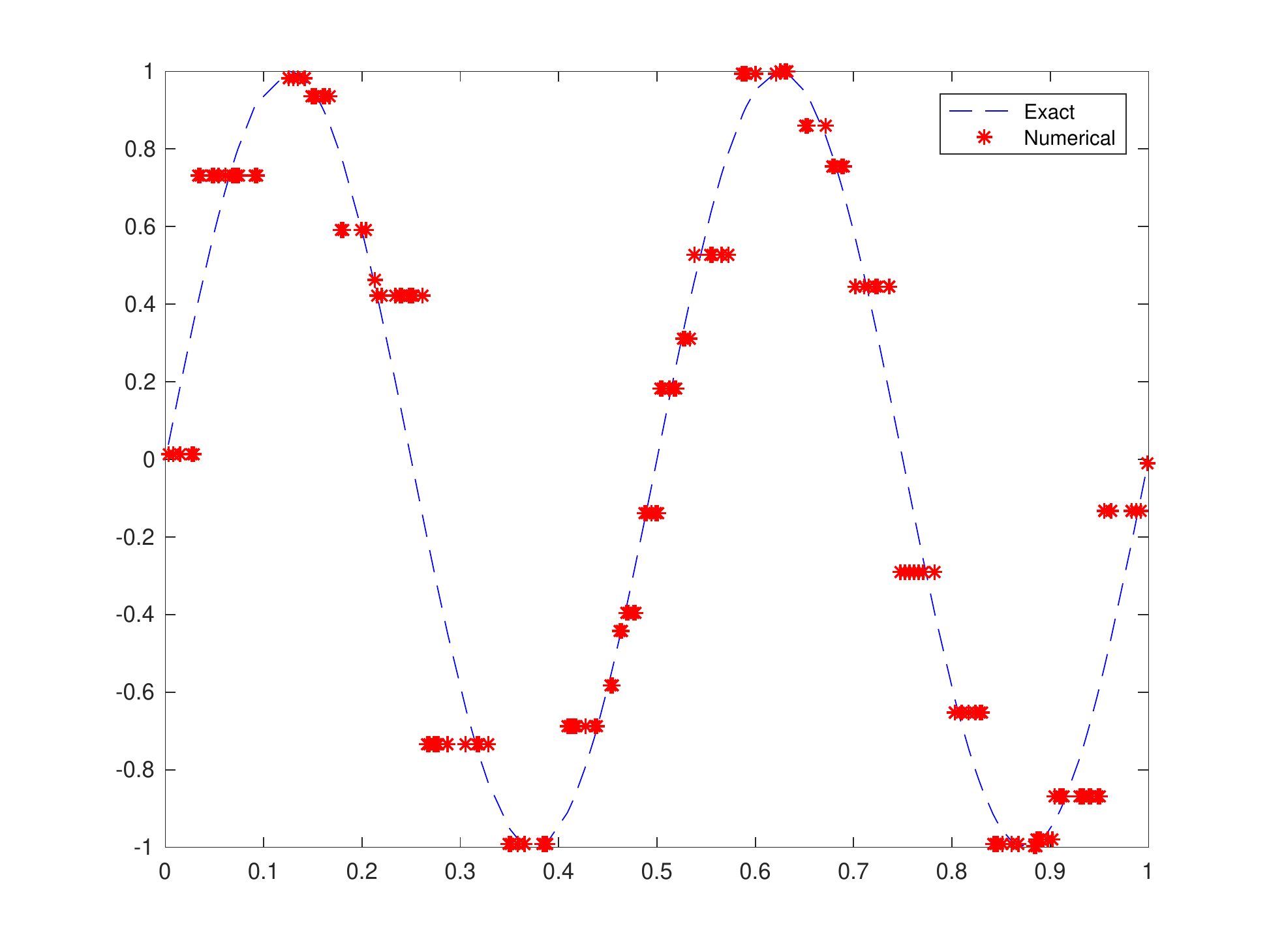}}
	%{\includegraphics[width=0.49\textwidth]{Figures/Error_1d_rand.eps}}
	\caption{Approximation of $f(x)=\cos(4\pi x)$ using $n=32$ random training
		data.}
	\label{fig:Cosine_1d}
\end{figure}

We then examine the errors in the FNN approximation with respect to
increasing number of the training data $n$. The errors are shown in
Fig. \ref{fig:Error_1d}, using both uniformly distributed training
data (left figure) and randomly generated training data (right
figure). The $n^{-1}$ error convergence rate is clearly visible,
consistent with the estimate in Theorem \ref{thm:error}.
\begin{figure}[htbp]
	\centering
	{\includegraphics[width=0.49\textwidth]{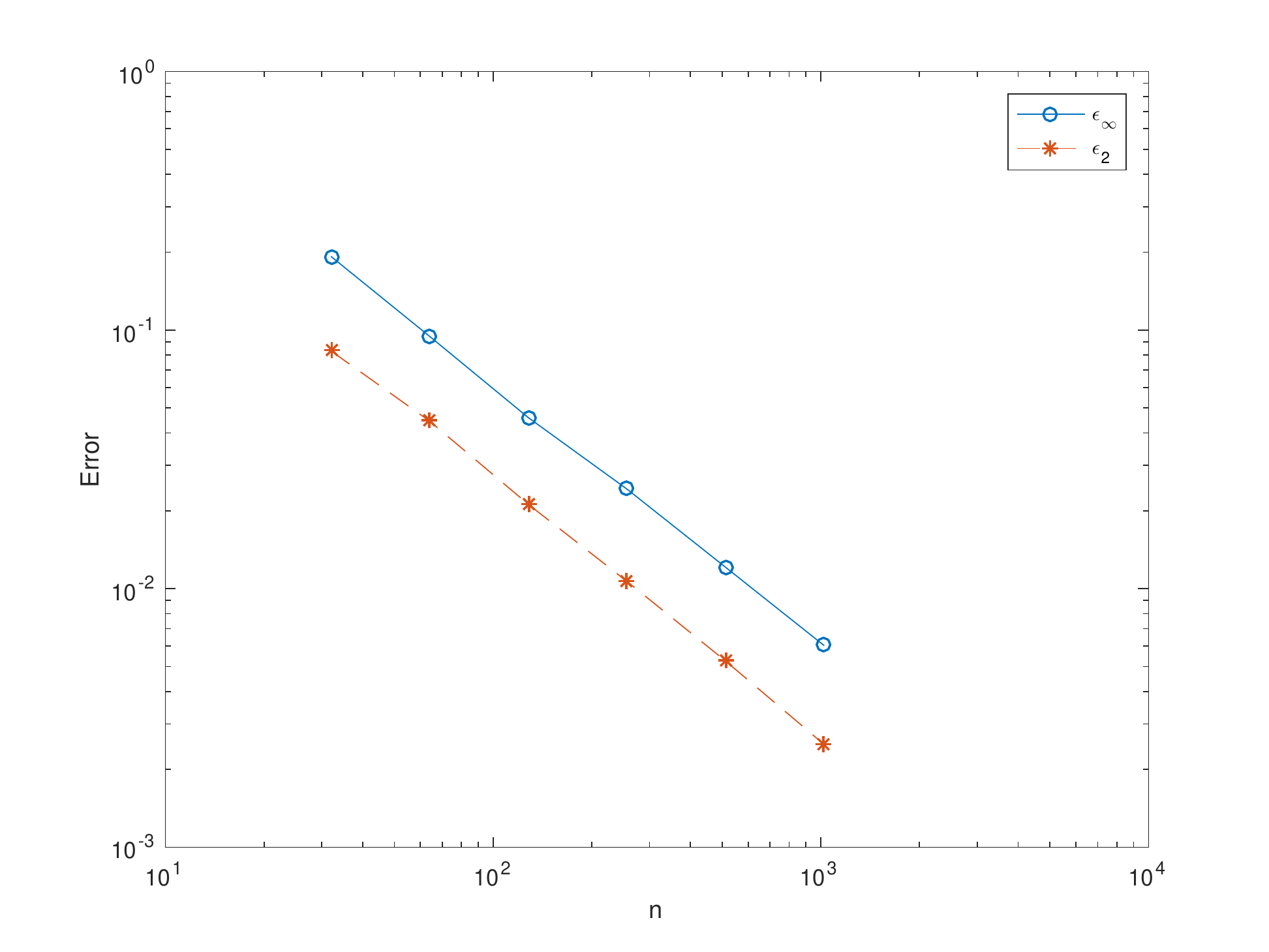}}
	{\includegraphics[width=0.49\textwidth]{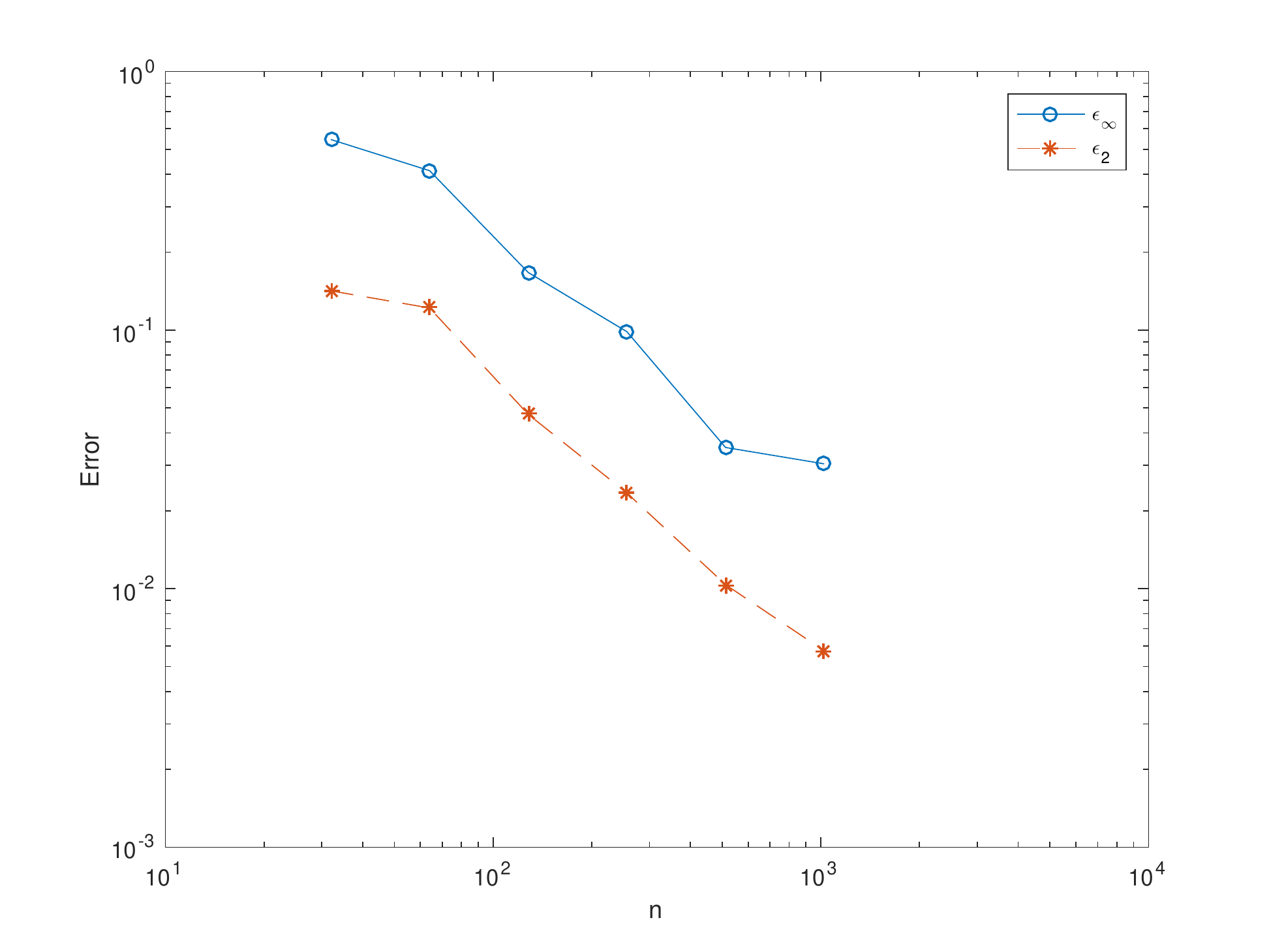}}
	\caption{Errors {\em vs.} $n$ for 1d cosine function. Left: uniform training data;
		Right: random training data.}
	\label{fig:Error_1d}
\end{figure}

Next we consider a univariate discontinuous function, which is rather
arbitrarily chosen as
$$
f(x) = 3s(x-0.313)+ s(x-0.747)+ 2\cos(4\pi x), \qquad x\in [0,1],
$$
where $s(x)$ is the step function \eqref{step}. 
The left of Fig. \ref{fig:1d_jump} shows the FNN approximation using
$n=1,024$ uniformly distributed training samples, whereas the right of Fig. \ref{fig:1d_jump}
shows the error decay with respect to the number of training
samples. We observe good approximation property and expected
convergence rate from \eqref{rate}.
\begin{figure}[htbp]
	\centering
	{\includegraphics[width=0.49\textwidth]{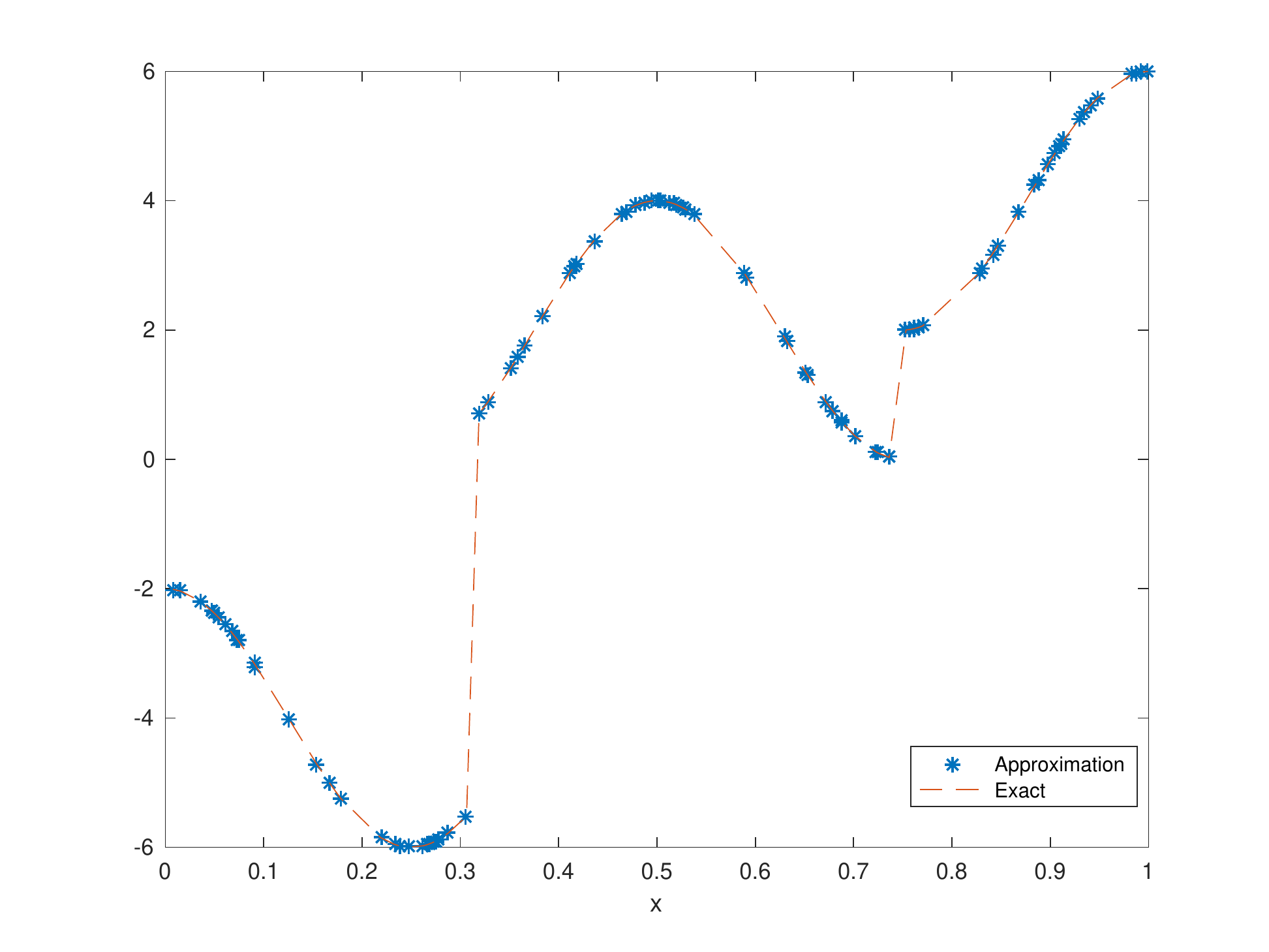}}
	{\includegraphics[width=0.49\textwidth]{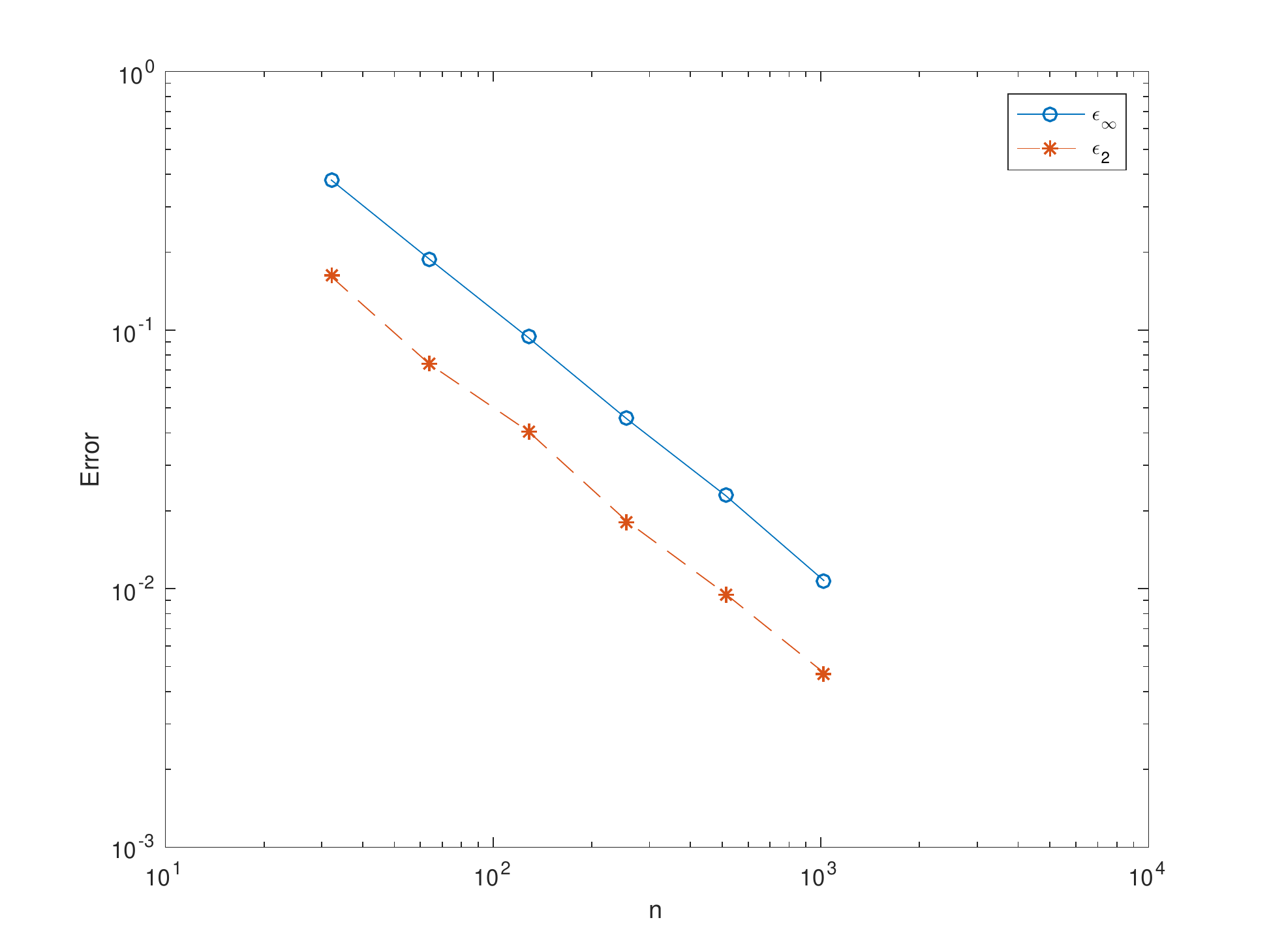}}
	\caption{Approximation of a discontinuous function in 1d with uniform
		training data. Left: Function
		approximation with $n=1,024$ training samples;
		Right: Errors {\em vs.} number of training samples $n$.}
	\label{fig:1d_jump}
\end{figure}

\subsection{Multivariate functions}

We now consider multivariate functions. 
%$f:[0,1]^d\to\R$, $d>1$. 
Throughout this section, all training samples are generated
randomly using uniform distribution. 
Two functions are
examined, a sine function and a Gaussian function, in the following form,
\begin{equation*} %\label{Multidfunc}
f_S(\x) = \sin\left( \omega \sum_{i=1}^d x_i\right),~~{\bf x} \in [0,1]^d, \quad
f_G ({\bf x}) = \exp\left(-\sum_{i=1}^d
\left(\frac{x_i}{2}\right)^2 \right),~~{\bf x} \in [-1,1]^d.
\end{equation*}

First, we consider the functions in 2-dimension ($d=2$).
%following 2-dimensional function,
%$$
%f(\x) = \sin(2\pi (x_1+x_2)), \qquad \x \in [0,1]^d.
%$$
The numerical errors induced by our FNN  for these two functions are plotted in
Fig. \ref{fig:2d}, with respect to the number of training
samples. We observe the expected rate of error convergence \eqref{rate}.
%
%\begin{figure}[htbp]
%	\centering
%	\includegraphics[width=0.618\textwidth]{Figures/2d_Sine.eps}
%	\caption{Approximation of $\sin(2\pi (x_1+x_2))$ with random
%		training data: Errors {\em vs.} number of training
%                samples.}
%	\label{fig:2dsine}
%\end{figure}
\begin{figure}[htbp]
	\centering
	{\includegraphics[width=0.49\textwidth]{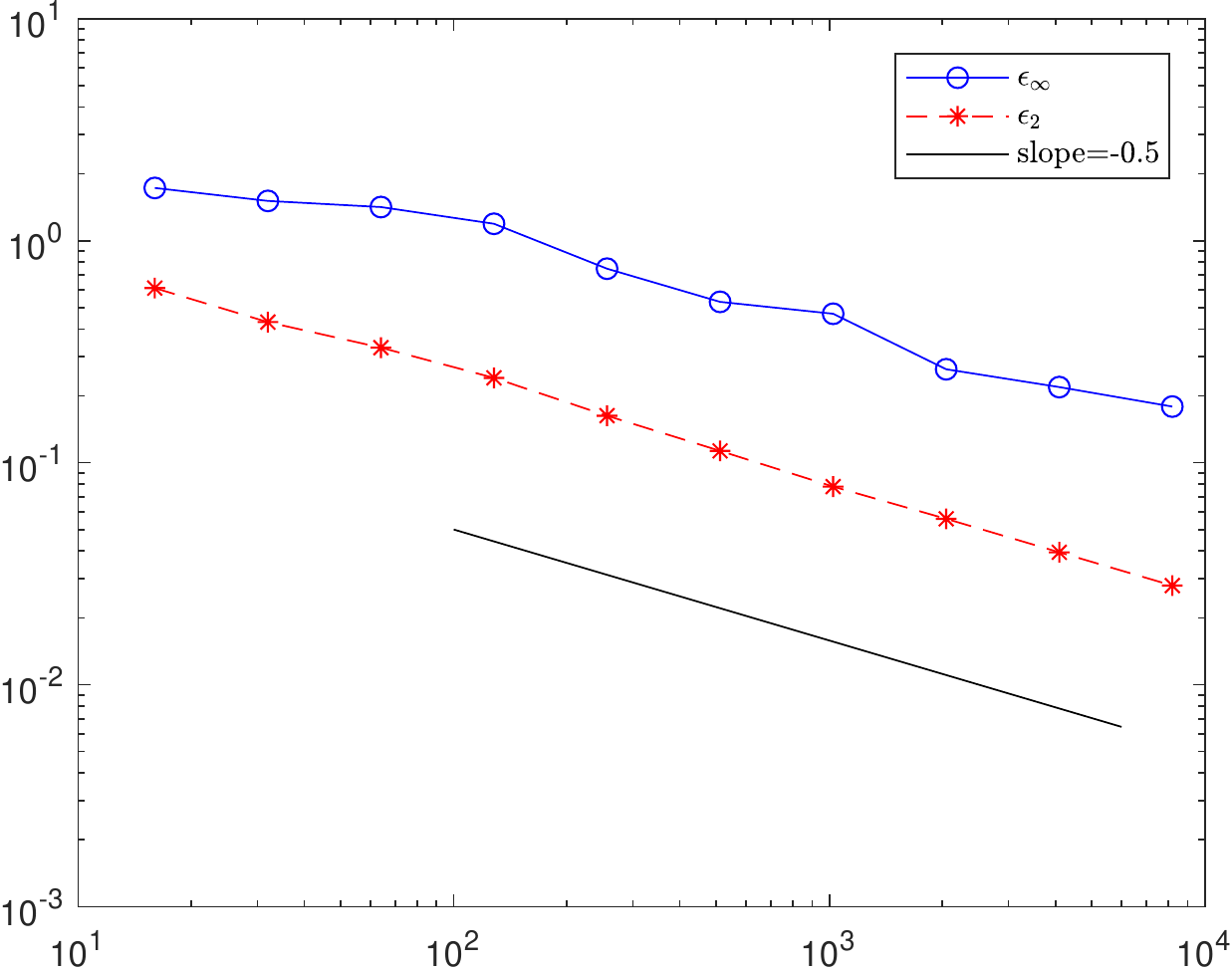}}
	{\includegraphics[width=0.49\textwidth]{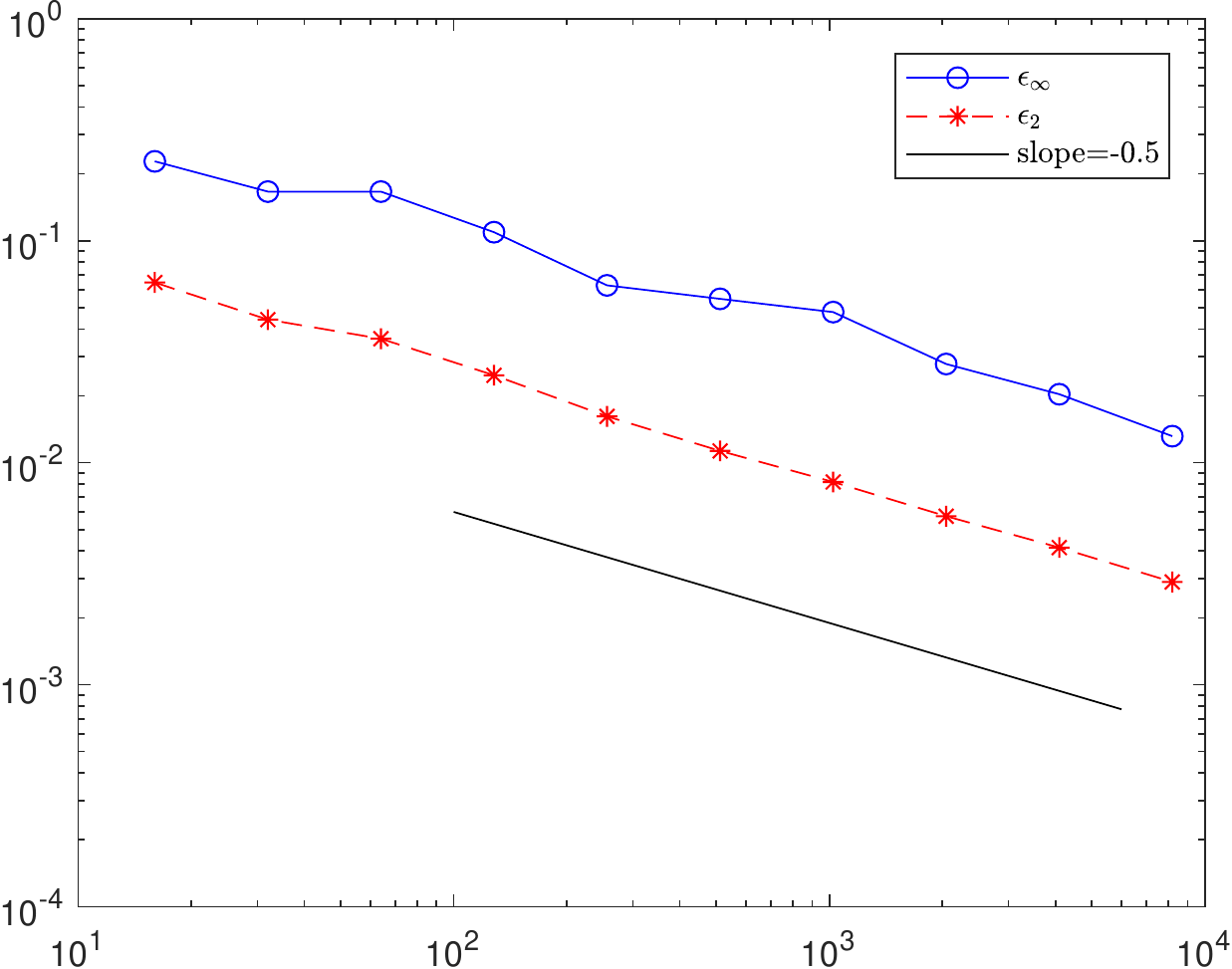}}
	\caption{Approximation errors of the sine and Gaussian functions 
		in $d=2$ versus number of training samples. Left: $f_S$ with $\omega=2\pi$;
		Right: $f_G$.}
	\label{fig:2d}
\end{figure}

Next, we consider functions in $d=4$ dimensions.
%Two functions are
%examined, a sine function and a Gaussian function, in the following form,
%\begin{equation} \label{4dfunc}
%f_S(\x) = \sin\left(\pi \sum_{i=1}^d x_i\right), \qquad
%f_G ({\bf x}) = \exp\left(-\sum_{i=1}^d
%  \left(\frac{x_i}{2}\right)^2 \right).
%\end{equation}
The numerical errors induced by our FNN are
plotted in Fig. \ref{fig:4d}, with respect to the number of training
samples. Again, we observe expected error behavior \eqref{rate}.
\begin{figure}[htbp]
	\centering
	{\includegraphics[width=0.49\textwidth]{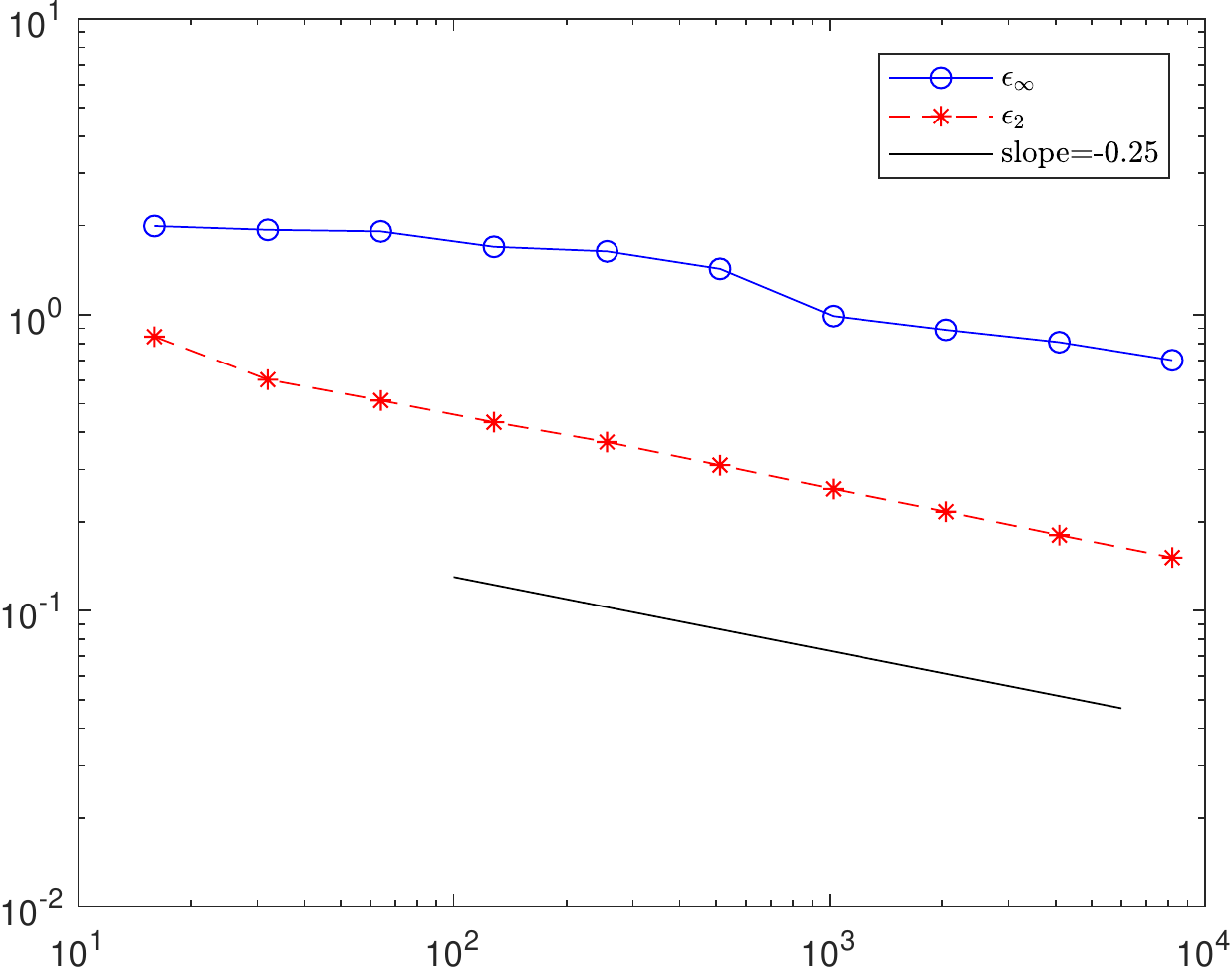}}
	{\includegraphics[width=0.49\textwidth]{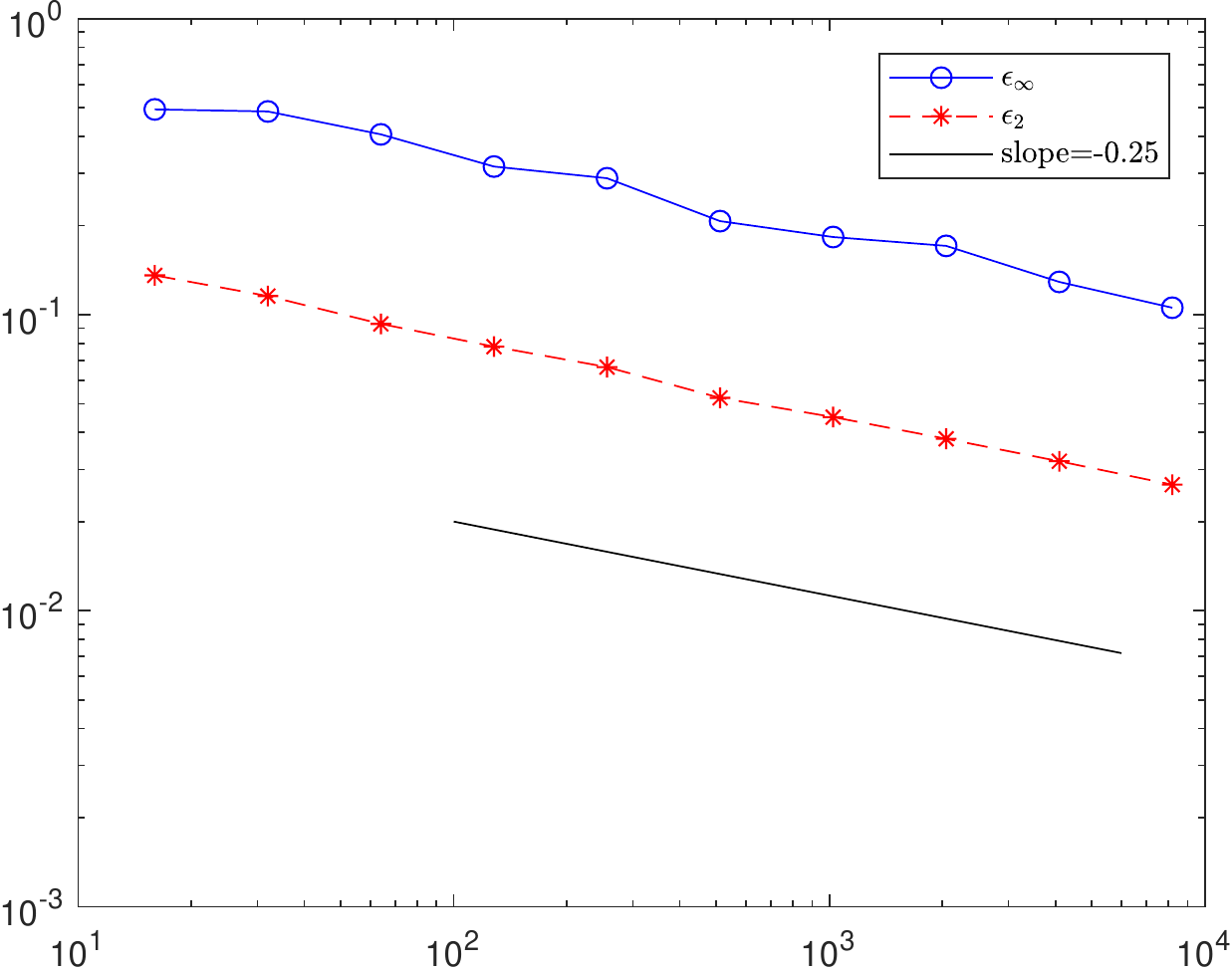}}
	\caption{Approximation errors of the sine and Gaussian functions 
		in $d=4$ versus number of training samples. Left: $f_S$ with  $\omega=\pi$;
		Right: $f_G$.}
	\label{fig:4d}
\end{figure}

%
%\begin{figure}[htbp]
%	\centering
%	{\includegraphics[width=0.49\textwidth]{Figures/4d_Sine_SIM.eps}}
%	{\includegraphics[width=0.49\textwidth]{Figures/4d_Gaussian_SIM.eps}}
%	\caption{Same as Fig. \ref{fig:4d} except for using formula \eqref{f_pw}.}
%	\label{fig:4d_s}
%\end{figure}
%
%
%\begin{figure}[htbp]
%	\centering
%	{\includegraphics[width=0.49\textwidth]{Figures/10d_Sine.eps}}
%	{\includegraphics[width=0.49\textwidth]{Figures/10d_Gaussian.eps}}
%	\caption{Approximation of $\sin(2\pi \sum_{i=1}^d x_i)$ (left)
%		and Gaussian function $a_i=1,\chi_i=0.5$ (right) in 10d with random
%		training data: Errors {\em vs.} $n$.}
%	\label{fig:10d}
%\end{figure}

%\begin{figure}[htbp]
%	\centering
%	{\includegraphics[width=0.49\textwidth]{Figures/10d_Sine_SIM.eps}}
%	{\includegraphics[width=0.49\textwidth]{Figures/10d_Gaussian_SIM.eps}}
%	\caption{Same as Fig. \ref{fig:10d} except for using formula \eqref{f_pw}.}
%	\label{fig:10d_s}
%\end{figure}

%\input Conclusion

\section{Summary} \label{sec:summary}

In this paper we presented a new explicit construction of feedforward
neural network (FNN). Our network consists of two hidden layers and
uses the step function as the activation function. It is able to
construct a Voronoi diagram of a given multivariate domain and
provides a piecewise constant approximation of any function defined
in the domain. Our construction uses $n^2$ neurons, where $n$ is the
number of training samples. It is worth noting that it is possible to
construct a similar piecewise constant approximation FNN using a
smaller number (less than $n^2$) of neurons. That, however, would require more complex
signal pathways among the neurons, with potential lateral connections
within the layers. Our current construction represents perhaps
the most straightforward construction of this type. In addition to
being another (new) constructive proof, from the mathematical view
point, for the universal approximation properties of FNNs, our
construction is also practical and can be easily adopted for applications.

%\input Appendix

%\section*{Acknowledgment}
%This work is in part supported by AFOSR, DOE, NNSA, NSF.

\bibliographystyle{plain}
\bibliography{neural}

\end{document}